\documentclass[12pt,reqno]{amsart}

\usepackage{amsaddr}
\usepackage{a4wide}
\usepackage{hyperref}
\usepackage{marginnote}
\usepackage{amsmath,amsthm,latexsym,xspace,amscd,amssymb,xy}               % AmSLaTeX
\usepackage{color}
\usepackage{enumerate}

\newtheorem{teo}{Theorem}[section]
\newtheorem{lema}[teo]{Lemma}
\newtheorem{cor}[teo]{Corollary}
\newtheorem{prop}[teo]{Proposition}

\theoremstyle{definition}
\newtheorem{defi}[teo]{Definition}

\newtheorem{remark}[teo]{Remark}

\DeclareMathOperator{\Kspan}{span}
\DeclareMathOperator{\parpar}{par}
\DeclareMathOperator{\Aut}{Aut}
\DeclareMathOperator{\Id}{Id}

\newcommand{\F}{{\mathbb F}}

\newcommand{\ob}{ \mathrm{ob} }
\newcommand{\mor}{ \mathrm{mor} }
\newcommand{\id}{ \mathrm{id} }

\newcommand{\de}{\delta}

\newcommand{\af}{\alpha}

\newcommand{\bt}{{\beta}}

\newcommand{\m}{{}^{-1}}

\def\ndv{\ {\mid \kern -0.7 em {\scriptstyle \not}} \ \ }

\def\nd{\ {\mid \kern -0.4 em {\scriptstyle \not}} \ \ }

\newtheorem{rem}{Remark}

\date{\today}

\title[Artinian and Noetherian Partial Skew Groupoid Rings]{Artinian and Noetherian Partial Skew Groupoid Rings}

\begin{document}

\author{Patrik Nystedt}
\address{University West,
Department of Engineering Science, 
SE-46186 Trollh\"{a}ttan, Sweden}

\author{Johan \"{O}inert}
\address{Blekinge Institute of Technology,
Department of Mathematics and Natural Sciences,
SE-37179 Karlskrona, Sweden}

\author{H\'{e}ctor Pinedo}
\address{Universidad Industrial de Santander, Escuela de Matem\'{a}ticas,
Carrera 27 Calle 9,
Edificio Camilo Torres
Apartado de correos 678,
Bucaramanga, Colombia}

\email{patrik.nystedt@hv.se; johan.oinert@bth.se; hpinedot@uis.edu.co}


\begin{abstract}
Let $\alpha = \{ \alpha_g : R_{g^{-1}} \rightarrow R_g \}_{g \in \mor(G)}$
be a partial action of a groupoid $G$ on a non-associative ring $R$
and let $S = R \star_{\alpha} G$ be the associated 
partial skew groupoid ring.
We show that if $\alpha$ is global and unital, 
then $S$ is left (right) artinian if and only if 
$R$ is left (right) artinian and $R_g = \{ 0 \},$
for all but finitely many $g \in \mor(G)$.
We use this result to prove that if $\alpha$ is unital and $R$ is alternative,
then $S$ is left (right) artinian
if and only if $R$ is left (right) artinian and $R_g = \{ 0 \},$
for all but finitely many $g \in \mor(G)$.
Both of these results apply to partial skew group rings, and in particular
they generalize a result by J. K. Park for classical skew group rings, i.e. the 
case when $R$ is unital and associative, and $G$ is a group
which acts globally on $R$. Moreover,
we provide two
applications of our main result.
Firstly, we generalize I. G. Connell's classical result for group rings
by giving a characterization of artinian (non-associative) groupoid rings.
This result is in turn applied to partial group algebras.
Secondly, we give a characterization of artinian Leavitt path algebras.
At the end of the article, we use globalization to 
analyse noetherianity and artinianity of partial skew
groupoid rings as well as establishing two Maschke-type
results, thereby generalizing results by Ferrero and Lazzarin
from the group graded case to the groupoid situation. 
\end{abstract}

\maketitle

\pagestyle{headings}

\section{Introduction}

In 1963, I. G. Connell \cite{C} showed that if $R$ is an associative and unital ring, and $G$ is a group,
then the group ring
$R[G]$
is left (right) artinian if and only if $R$ is left 
(right) artinian and $G$ is finite.
Later on, D. S. Passman gave examples of artinian twisted group rings by infinite groups (see \cite[Section 4]{passman1971}).
Passman's examples show that Connell's result can not be generalized
to twisted group rings or, more generally, crossed products.

Another type of crossed products, generalizing group rings, are the skew group rings.
Recall that if $\alpha : G \ni g \mapsto \alpha_g \in \Aut(R)$ 
is a group homomorphism from $G$
to $\Aut(R)$, the group of ring automorphisms of $R$,
then the \emph{skew group ring} $R *_{\alpha} G$ is the set 
of finite formal sums of the form $\sum_{g \in G} r_g g$
with addition defined componentwise and multiplication
defined by the relations $(r g) (s h) = (r \alpha_g(s)) gh$,
for $r,s \in R$ and $g,h \in G$.
In 1979, J. K. Park \cite{P} generalized Connell's result to skew group rings
by showing the following.

\begin{teo}[Park \cite{P}]\label{ParkTheorem}
If $R$ is a unital and associative ring, and $\alpha$
is a group homomorphism from a group $G$ to $\Aut(R)$,
then the skew group ring $R *_{\alpha} G$ is left (right) artinian 
if and only if $R$ is left (right) artinian and $G$ is finite. 
\end{teo}

In this article, we consider two generalizations (see Theorem \ref{ParkGeneral} and Theorem \ref{ParkGeneralAlternative}) of 
Theorem \ref{ParkTheorem} in the context of partial skew
groupoid rings over non-associative rings, i.e. rings which are not necessarily associative.
Previously, partial skew groupoid rings have been defined only over 
associative rings. However, since there are many interesting
examples of non-associative rings with various types of
actions, it is only natural to seek such a theory in this more 
general sense. For instance, our Theorem \ref{ParkGeneral} 
holds when $R$ equals any of the algebras in the infinite chain of classical
Cayley-Dickson doublings: the real numbers $\mathbb{R}$,
the complex numbers $\mathbb{C}$, Hamilton's quaternions $\mathbb{H}$,
Graves' octonions $\mathbb{O}$, the sedenions $\mathbb{S}$,
the trigintaduonions $\mathbb{T}$ etc.
Other important classes of examples to which our
Theorem~\ref{ParkGeneral} can be applied comes from the cases when $R$ is a Jordan algebra or a Baric algebra.

The notion of a partial action of a group on a C*-algebra was introduced by R. Exel
\cite{Exel94}, as an efficient tool to their study.  Since then, the theory of (twisted) partial actions on C*-algebras has
played a key role in the characterization of several classes of C*-algebras as crossed products by (twisted) partial actions,
e.g. AF-algebras \cite{Exel95}, Bunce-Deddens algebras \cite{Exel94BD}, Cuntz-Krieger algebras \cite{ExelLaca99} and Cuntz-Li algebras \cite{BoavaExel13}, (see also the survey \cite{D2}).
In a purely algebraic context, partial skew group rings were 
introduced by M. Dokuchaev and R. Exel \cite{DokuchaevExel05}
as a generalization of classical skew group rings
and as an algebraic analogue of partial crossed product C*-algebras.

Partial group actions can 
be
described in terms of premorphisms, which is a notion introduced by McAlister and Reilly \cite{MR}.
As explained
in \cite{KL}, a partial action of a group $G$ on a set $X$ is a unital premorphism from $G$ to the inverse monoid  $\mathcal{I}(X)$, consisting of bijections between subsets of $X$. This
perspective
motivated to the study of partial actions of 
other algebraic structures rather than groups on sets,  such as semigroups \cite{GouHol1, GouHol2, Hol1, KU, MS}, and  ordered groupoids  \cite{BFP,Gi}. 
Groupoids also appear naturally in the context of partial group actions. Indeed, in \cite{AB} the author constructs a functor from the category of partial actions
to the category of groupoids.
  In \cite{BFP,BP}, partial actions of  groupoids on rings
and the corresponding partial skew groupoid rings were introduced, and recently the authors of \cite{GY} gave a description of Leavitt Path algebras as partial skew groupoid rings.

Recall that  a {\it groupoid}   $G$ is  a small category
with the property that all its morphisms are isomorphisms.
The family of objects and morphisms
of $G$ will be denoted by $\ob(G)$ and $\mor(G)$ respectively. 
As usual one identifies and object $e$ with the identity morphism 
$\Id_e$, so $\ob(G) \subseteq \mor(G)$.
If $g \in \mor(G)$, then the domain and codomain of 
$g$ will be denoted by $d(g)$ and $c(g)$, respectively.
We let $G^2$ denote the set of all pairs $(g,h) \in \mor(G) \times \mor(G)$
that are composable, i.e. such that $d(g)=c(h)$.
Let $R$ be a {\it non-associative ring}.
By this we mean that $R$ is an additive group equipped
with a multiplication which is distributive with respect to addition.
If $R$ is unital, then the multiplicative identity is denoted by $1_R$
and is always assumed to be non-zero.
The identity map $R \rightarrow R$ is denoted by $\id_R$.
Recall from \cite{BP} that
$\alpha = \{ \alpha_g : R_{g^{-1}} \rightarrow R_g \}_{g \in \mor(G)}$
is called a {\it partial action of $G$ on $R$} if for each $g \in \mor(G)$, 
$R_g$ is an ideal of $R_{c(g)}$, $R_{c(g)}$ is an ideal of $R$ and 
$\alpha_g : R_{g^{-1}} \rightarrow R_g$ is a ring isomorphism 
satisfying the following three axioms:
\begin{itemize}
\item[(P1)] if $e \in \ob(G)$, then $\alpha_e = \id_{R_e}$;
\item[(P2)] if $(g,h) \in G^2$, then $R_{(gh)^{-1}} \supseteq 
\alpha_h^{-1} ( R_h \cap R_{g^{-1}} )$;
\item[(P3)] if $(g,h) \in G^2$ and $x \in \alpha_h^{-1} ( R_h \cap R_{g^{-1}} )$,
then $(\alpha_g \circ \alpha_h)(x) = \alpha_{gh}(x)$.
\end{itemize}
The associated {\it partial skew groupoid ring} $R \star_\af G$ is the set 
of all finite formal sums $\sum_{g \in \mor(G)} r_g \de_g$, 
where $r_g \in R_g$, with addition defined componentwise and  
multiplication determined by the rule 
\begin{equation}\label{multiplicationrule}
(r_g \delta_g) (r_h' \delta_h) = \alpha_g( \alpha_{g^{-1}}(r_g) r_h'  ) \delta_{gh},
\end{equation}
if $(g,h) \in G^2$, and $(r_g \delta_g) (r_h' \delta_h) = 0$, otherwise.
Since the ring structure of $R \star_{\alpha} G$
only depends on the choice of the rings $R_e$, for $e \in \ob(G)$, 
we may take $R$ to be {\it any} ring having these rings as ideals.
From this point of view, we may therefore assume that the following fourth axiom holds:
\begin{itemize}
\item[(P4)] $R = \oplus_{e \in \ob(G)} R_e$.
\end{itemize}
By adding this fourth axiom we get another advantage.
Namely, our definition of partial groupoid actions,
in the case when $G$ is a group, i.e. when $G$ has one object,
now coincides with the classical definition of partial group actions on rings.
We say that $\alpha$ is {\it unital} if every non-zero $R_g$ is unital, for $g\in \mor(G)$.
The action $\af$ is called {\it global}, 
if $\af_{gh}=\af_g\af_h$, for $(g,h) \in G^2$. 
It follows from \cite[Lemma 1.1(i)]{BP} that $\af$ is global, 
if and only if $R_g=R_{c(g)}$, for all $g\in \mor(G)$.
In that case, apart from the associativity requirement on $R$, $R \star_{\alpha} G$ coincides
with the definition of a {\it skew groupoid ring} given in \cite{lundstromoinert2010}.
Here is an outline of the article.

In Section \ref{preliminaries}, we
recall some notions and results from 
non-associative ring theory that we need 
in the sequel.
In Section \ref{partialskewgroupoidrings}, 
we show the following generalizations of Theorem \ref{ParkTheorem}.

\begin{teo}\label{ParkGeneral}
If $\alpha$ is a global unital action of a groupoid $G$ on a
non-associative ring $R$,
then the partial skew groupoid ring
$R \star_{\alpha} G$ is left (right) artinian, if and only if,
$R$ is left (right) artinian
and $R_g = \{ 0 \},$ for all but finitely many $g \in \mor(G)$.
\end{teo}

\begin{teo}\label{ParkGeneralAlternative}
If $\alpha$ is a unital partial action of a groupoid $G$ 
on an alternative ring $R$, 
then the partial skew groupoid ring
$R \star_{\alpha} G$ is left (right) artinian, if and only if ,
$R$ is left (right) artinian
and $R_g = \{ 0 \},$ for all but finitely many $g \in \mor(G)$.
\end{teo}

A couple of remarks are needed here.
First, notice that the condition that $R_g = \{ 0 \}$ for all but
finitely many $g \in \mor(G)$ can not, in general, be replaced by the statement
that $\mor(G)$ is finite in
Theorem \ref{ParkGeneral}, 
nor in Theorem \ref{ParkGeneralAlternative}.
To see this,
let $G$ be any infinite groupoid and let $K$ be any 
left (or right) artinian non-associative ring, e.g. a field.
Fix $e \in \ob(G)$ and put $R_e = K$.
If $g \in \mor(G)$ and $g \neq e$, then put $R_g = \{ 0 \}$.
Let $\alpha_e = \id_K$ and for $g \in \mor(G)$ with
$g \neq e$, then let $\alpha_g : R_{g^{-1}} \rightarrow R_g$
be the zero map. Then $R \star_{\alpha} G = K$ which is 
left (or right) artinian even though $\mor(G)$ is infinite.

Secondly, it follows by \cite[Theorem 2.1]{BP}  that  any unital partial groupoid  $\alpha$ action on a unital ring $R$ admits a globa\-lization  $\beta$ on a ring $T$,  provided that $R$ is associative.  But according to \cite[Example 1.4]{FL} one can not guarantee that $T$ is unital, not even in the group case. Therefore, a Morita equivalence can in general not be used to deduce Theorem \ref{ParkGeneralAlternative} from Theorem \ref{ParkGeneral} (see Corollary \ref{art}).

In the following two sections  we give some applications of our main results.
Namely, in Section \ref{sec:AppMix} we generalize Connell's result
and provide a characterization of left (right) artinian \emph{groupoid rings} (see Theorem \ref{groupoidring}).
We also characterize left (right) artinian \emph{generalized matrix rings} (see Corollary \ref{cor:genmatrixring})
and show an analogue of Connell's result for \emph{partial group algebras} (Corollary \ref{cor:PGA})

In Section \ref{sec:LPA} we recall the definition
of a \emph{Leavitt path algebra}
and briefly explain how it can be viewed as a partial skew group ring, using \cite{GR}.
Thereafter, we apply our main results and provide a characterization
of artinian (and semisimple) Leavitt path algebras (see Theorem \ref{LPAthm}).

In Section \ref{sec:Morita},
we use globalization and Morita 
equivalence to deduce necessary and sufficient conditions
for the artinianity and noetherianity of partial skew groupoid rings,
as well as establishing two Maschke-type results (see Theorem~\ref{teo:Maschke} and Corollary~\ref{cor:Maschke}).

\section{Preliminaries on Rings and Modules}\label{preliminaries}

In this section, we recall some results from non-associative
ring theory that we need in the sequel.
Throughout this section, let $R$ and $S$ denote
non-associative rings with $R$ a subring of $S$.
Recall that the center of $R$, denoted by $Z(R)$, is
the set of elements in $R$ that commute and associate 
with all elements of $R$. In other words, an element
$x \in R$ belongs to $Z(R)$ precisely when for every
choice of $r,r' \in R$, we have that 
$xr = rx$, $(xr)r' = x(rr')$, $(rx)r' = r(xr')$ and $(rr')x = r(r'x)$. 

\begin{prop}\label{Zfield}
If $R$ is unital and simple, then $Z(R)$ is a field.
\end{prop}

\begin{proof}
See e.g. \cite[Proposition 9]{NOR2015}.
\end{proof}

By a \emph{left $R$-module} we mean an additive group $M$
equipped with a biadditive map 
$R \times M \ni (r,m) \mapsto rm \in M$.
In that case, if $R$ is unital and $1_R m = m$, for $m \in M$,
then we say that $M$ is \emph{unital as a left $R$-module}.
By a \emph{left $R$-submodule} of $M$ we mean an additive subgroup $N$
of $M$ such that if $n \in N$ and $r \in R$, then $rn \in N$.
Recall that $M$ is called \emph{artinian} if it satisfies the 
descending chain condition on its poset of submodules.
The ring $R$ is called \emph{left artinian} if it is artinian
as a left module over itself.
The concepts of right module and right artinian ring 
are defined analogously.
If the ring $R$ is both left and right artinian, then it is said to be \emph{artinian}.

\begin{prop}\label{subringartinian} 
Suppose that $M$ is a left (right) $S$-module.
If $M$ is artinian as a left (right) $R$-module, 
then $M$ is artinian as a left (right) $S$-module.
\end{prop}

\begin{proof}
This follows immediately from the fact that $R\subseteq S$.
\end{proof}

\begin{prop}\label{lemmasummand}
Suppose that $R$ is a direct summand of the left (right) $R$-module $S$.
Then, for any right (left) ideal $I$ of $R$,
the relation $IS \cap R = IR $ ($SI \cap R = RI$) holds.
In particular, if $S$ is right/left artinian (noetherian) 
and $R$ is unital, then $R$ is right/left artinian (noetherian).
\end{prop}

\begin{proof}
Let $T$ be a left (right) $R$-module such that
$S = R \oplus T$ as left (right) $R$-modules. 
Let $I$ be a right (left) ideal of $R$. Then $IS \cap R = 
[ I(R \oplus T) ] \cap R = 
[ IR \oplus IT ] \cap R = 
IR$, $(SI \cap R = RI)$. In particular, if $R$ is unital,
then we get $IR=I$ ($RI=I$) from which the last part follows.
\end{proof}

\begin{prop}\label{exactsequence}
Let $R$ be a unital ring.
Suppose that $M$ is a left (right) $R$-module
and that $N$ is a left (right) $R$-submodule of $M$.
Then $M$ is artinian (noetherian), if and only if, 
$N$ and $M/N$ are artinian (noetherian).
\end{prop}

\begin{proof}
See \cite[Proposition 0.2.19]{rowen1998}.
\end{proof}

\begin{prop}\label{directsumartinian}
Let $R$ be a unital ring.
Suppose that $M$ is a left (right) $R$-module.
If $M_1,\ldots,M_n$ are left (right) $R$-submodules of $M$
such that $M = M_1 \oplus \cdots \oplus M_n$,
then $M$ is artinian (noetherian) if and only if for each 
$i \in \{ 1,\ldots,n \}$, $M_i$ is artinian (noetherian).
\end{prop}

\begin{proof}
We can use Proposition \ref{exactsequence}.
The ``only if'' statement is clear.
For the ``if'' statement, consider the case $i=2$ and
notice that $M_1$ is a left (right) $R$-submodule of $M_1 \oplus M_2$ and 
that $(M_1 \oplus M_2)/M_1 \cong M_2$.
The general case now follows by induction over $i$.
\end{proof}

The term {\it ideal}
refers to \emph{two-sided ideal},
unless otherwise stated.
The ring $R$ is called \emph{simple} if $\{ 0 \}$ and $R$
are the only ideals of $R$.
An ideal $I$ of $R$ is called \emph{maximal} if 
$I \subsetneq R$ and for every ideal $J$ of $R$
with $I \subseteq J \subsetneq R$, the relation $I = J$ holds.
Note that an ideal $I$ of $R$ is maximal if and only if $R/I$ is simple.
The ring $R$ is called {\it semisimple} if it is a finite direct sum
of simple rings.
Following \cite{C},
the intersection of all maximal ideals of $R$,
denoted by ${\mathcal S}(R)$,
will be called the {\it simplicial radical} of $R$.

\begin{prop}\label{semisimplicitySR}
If $R$ is unital and artinian, then 
$R / {\mathcal S}(R)$ is semisimple.
\end{prop}

\begin{proof}
See \cite[Theorem 1.7 on p. 6]{elduque1994}.
\end{proof}

Recall that $R$ is called {\it alternative} if for all $x,y \in R$,
the relations $x^2 y = x(xy)$ and $x y^2 = (xy)y$ hold.
If $R$ is alternative one may define a radical of $R$, denoted
${\mathcal J}(R)$, similar to the Jacobson radical from the associative situation.
In the alternative setting ${\mathcal J}(R)$ is often called
the Zhevlakov radical of $R$ (see \cite[p. 210]{zhevlakov1982}).
In that case, ${\mathcal J}$ is {\it hereditary} in the 
following sense.

\begin{prop}\label{JIhereditary}
If $R$ is artinian and $I$ is an ideal of $R$,
then ${\mathcal J}(I) = I \cap {\mathcal J}(R)$.
\end{prop}

\begin{proof}
See \cite[Theorem 3 on p. 204]{zhevlakov1982}.
\end{proof}

\begin{prop}\label{semisimplicityJR}
If $R$ is a unital alternative left (right) artinian ring,
then ${\mathcal S}(R) = {\mathcal J}(R)$ and
$R / {\mathcal J}(R)$ is semisimple.
\end{prop}

\begin{proof}
For the equality ${\mathcal S}(R) = {\mathcal J}(R)$,
see \cite[Proposition 1.2(ii) on p. 94]{elduque1994}.
The semisimplicity of $R / {\mathcal J}(R)$ now follows 
from Proposition \ref{semisimplicitySR}
(or \cite[Corollary on p. 250]{zhevlakov1982}).
\end{proof}

\begin{rem}
The proofs of Propositions \ref{semisimplicitySR}-\ref{semisimplicityJR}, as they appear in \cite{elduque1994} and \cite{zhevlakov1982}, presuppose that $R$ is an algebra over a field. However, it is clear from those proofs that the algebra structure of $R$ is not needed, and therefore one can suppose that $R$ is only a ring.
\end{rem}

\section{Partial Skew Groupoid Rings}\label{partialskewgroupoidrings}

In this section, we show Theorem \ref{ParkGeneral} and Theorem \ref{ParkGeneralAlternative}. 
We assume that 
$G$ is a groupoid, that $R$ is a non-associative ring and
that $\alpha = \{ \alpha_g : R_{g^{-1}} \rightarrow R_g \}_{g \in \mor(G)}$
is a partial action of $G$ on $R$.

\begin{prop}\label{IF}
If $R$ is left (right) artinian, and $R_g = \{ 0 \}$
for all but finitely many $g \in \mor(G)$,
then $R \star_{\alpha} G$ is left (right) artinian.
\end{prop}

\begin{proof}
Put $S = R \star_{\alpha} G$ and $S_0 = \oplus_{e \in \ob(G)} R_e \delta_e$.
Since $R$ is an artinian ring we get by  Proposition \ref{exactsequence} that  $R_{c(g)}$ is also an  artinian ring, the same
proposition implies that
that $R_g \delta_g$ is 
artinian 
as a left (right) module over $R_{c(g)} \delta_{c(g)}$ 
($R_{d(g)} \delta_{d(g)}$).
Therefore, for each $g \in \mor(G)$, we get, from Proposition \ref{subringartinian},
that $R_g \delta_g$ is 
artinian 
as a left (right) $S_0$-module.
Then  Proposition \ref{directsumartinian} implies that  
$S$ is artinian as a left (right) $S_0$-module.
Thus, by Proposition \ref{subringartinian}, 
the ring $S$ is left (right) artinian.
\end{proof}

\begin{defi}\label{Gstack}
A {\it subgroupoid} of $G$ is a subcategory 
of $G$ that is a groupoid in itself.
Define the subgroupoid $G^{\#}$ of $G$ in the following way.
The objects of $G^{\#}$ are all $e \in \ob(G)$ with $R_e$ non-zero.
The morphisms of $G^{\#}$ are all morphisms $g \in \mor(G)$
with $d(g),c(g) \in \ob(G^{\#})$.
Put $R^{\#} = \oplus_{e \in \ob( G^{\#} )} R_e$ and
let $\alpha^{\#}$ denote the restriction of $\alpha$ to $G^{\#}$.
Then $R \star_{\alpha} G = R^{\#} \star_{\alpha^{\#}} G^{\#}$.
\end{defi}

\begin{prop}\label{noetherianfinite}
If $R \star_\af G$ is left/right artinian (noetherian), 
then ${\rm ob}(G^{\#})$ is finite.
\end{prop}

\begin{proof}
Put $S = R \star_\af G$. 
Seeking a contradiction, suppose that there is
an infinite set of different elements $\{ e_i \}_{i \in {\mathbb N}}$ in $\ob(G)$
such that for every $i \in {\mathbb N}$, there is a non-zero $r_i \in R_{e_i}$. 

Suppose that $S$ is left artinian.
Define 
a set of left ideals $\{ I_i \}_{i \in {\mathbb N}}$
of $S$ by $I_i = \oplus_{g \in G^i} R_g \delta_g$, for $i \in {\mathbb N}$,
where $G^i = \{ g \in \mor(G) \mid d(g) \in \{ e_i, e_{i+1},e_{i+2},\ldots \} \}$.
For every $i \in {\mathbb N}$ and $r_i\in R_{e_i}$, we get that
$r_i \delta _{e_i} \in I_i \setminus I_{i+1}$. Hence, $\{ I_i \}_{i \in {\mathbb N}}$
is a strictly descending chain of left ideals of $S$,
which is a contradiction.

Suppose that $S$ is right artinian.
Define 
a set of right ideals $\{ I_i \}_{i \in {\mathbb N}}$
of $S$ by $I_i = \oplus_{g \in G^i} R_g \delta_g$, for $i \in {\mathbb N}$,
where $G^i = \{ g \in \mor(G) \mid c(g) \in \{ e_i, e_{i+1},e_{i+2},\ldots \} \}$.
For each $i \in {\mathbb N}$ and $r_i\in R_{e_i}$,  we get that
$r_i \delta _{e_i} \in I_i \setminus I_{i+1}$. Hence, $\{ I_i \}_{i \in {\mathbb N}}$
is a strictly descending chain of right ideals of $S$,
which is a contradiction.

Suppose that $S$ is left noetherian.
Define the set of left ideals $\{ I_i \}_{i \in {\mathbb N}}$
of $S$ by $I_i = \oplus_{g \in G^i} R_g \delta_g$, for $i \in {\mathbb N}$,
where $G^i = \{ g \in \mor(G) \mid d(g) \in \{ e_1, e_2,\ldots , e_i \} \}$.
For every $i \in {\mathbb N}$, we get that
$r_i e_{i+1} \in I_{i+1} \setminus I_i$. Hence, $\{ I_i \}_{i \in {\mathbb N}}$
is a strictly ascending chain of left ideals of $S$,
which is a contradiction.

Suppose that $S$ is right noetherian.
Define the set of right ideals $\{ I_i \}_{i \in {\mathbb N}}$
of $S$ by $I_i = \oplus_{g \in G^i} R_g \delta_g$, for $i \in {\mathbb N}$,
where $G^i = \{ g \in \mor(G) \mid c(g) \in \{ e_1, e_2 ,\ldots, e_i \} \}$.
For every $i \in {\mathbb N}$, we get that
$r_{i+1} e_{i+1} \in I_{i+1} \setminus I_i$. Hence, $\{ I_i \}_{i \in {\mathbb N}}$
is a strictly ascending chain of right ideals of $S$,
which is a contradiction.
\end{proof}

\begin{prop}\label{ONLYIF1}
If $\alpha$ is unital and $R \star_{\alpha} G$ is left/right artinian (noetherian),
then $R$ is left/right artinian (noetherian).
\end{prop}

\begin{proof}
Put $S = R \star_{\alpha} G$ and $S_0 = \oplus_{e \in \ob(G)} R_e \delta_e$.
By Proposition \ref{noetherianfinite}, 
we can write $S_0 = \oplus_{i=1}^n R_{e_i} \delta_{e_i},$
for $e_1,\ldots,e_n \in \ob(G)$ such that $R_{e_1},\ldots,R_{e_n}$ are all non-zero.
Then  $S_0$ is unital with multiplicative identity
given by $\sum_{i=1}^n 1_{R_{e_i}} \delta_{e_i}$.
Since $S_0$ is a direct summand
in $S$ as left/right $S_0$-modules, we get, from Proposition \ref{lemmasummand},
that $S_0$ is left/right artinian (noetherian). 
The desired conclusion now follows from (P4).
\end{proof}

Take $e,f \in \ob(G)$. We let 
$G(e,f)$ denote the set of morphisms in $G$ 
with domain $e$ and codomain $f$.
Then the set $G(e,e)$ is a group. We denote this group by $G_e$. Notice that $\af_e=\{\af_g\colon R_{g\m}\to R_g\}_{g\in G_e}$ is a partial (group) action of $G_e$ on $R_e$.

\begin{prop}\label{propfinitegroupoid}
The set  $\mor(G)$ is finite,
if and only if, $\ob(G)$ is finite and for each $e \in \ob(G)$,
the group $G_e$ is finite. 
\end{prop}

\begin{proof}
The ``only if'' statement is clear.
Now we show the ``if'' statement.
Suppose that $\ob(G)$ is finite and that for each $e \in \ob(G)$,
the group $G_e$ is finite. 
Since $\mor(G)=\bigcup_{e,f\in \ob(G)}G(e,f)$,  one only needs 
to show that $G(e,f)$ is finite, for all $e,f\in\ob(G). $ 
Notice that 
if $G(e,f)$ is non-empty, then we have a bijection $G_e\ni g\mapsto hg\in G(e,f)$, 
where $h$  is a fixed element in $G(e,f)$.
Hence, $G(e,f)$ is finite as desired.
\end{proof}

\begin{defi}
Recall  that a  subring $A$ of $R$ is called {\it $G$-invariant}, if 
for every $g \in {\rm mor}(G)$, the inclusion 
$\af_g(A\cap R_{g\m})\subseteq A\cap R_g$ holds.
In that case, the restriction of $\af_g$ to $R_{g^{-1}}' := R_{g^{-1}} \cap A$, 
for $g\in \mor(G)$, gives rise to a partial action of $G$ on $A$.
\end{defi}

\begin{prop}\label{lemmaskewgroupoidring}
If $\alpha$ is global and unital
such that $R \star_{\alpha} G$ is left (right) artinian and
for each $e \in \ob(G^{\#})$, $R_e$ is simple, then $\mor(G^{\#})$ is finite.
\end{prop}

\begin{proof}
Put $S = R \star_{\alpha} G = R^{\#} \star_{\alpha^{\#}} G^{\#}$
and take $e \in \ob(G)$ with $R_e$ non-zero.
Put $F = Z(R_e)$. Then $F$ is $G_e$-invariant and 
since $R_e$ is simple, we get, by Proposition \ref{Zfield}, that $F$ is a field.
Let $\beta$ be the induced group action of $G_e$ on $F$.
Put $T = F \star_{\beta} G_e$. 
We claim that $T$ is a direct summand in $S$ as right (left) $T$-modules.
Assume, for a moment, that the claim holds.
From Lemma \ref{directsumartinian} it follows that
$T$ is left (right) artinian. 
By Theorem \ref{ParkTheorem}, we get that $G_e$ is finite.
From Proposition \ref{noetherianfinite} and
Proposition \ref{propfinitegroupoid}, we have that $\mor(G^{\#})$ is finite.
Now we show the claim.
Since $F$ is a field, there is an $F$-subspace $V$
of $R_e$ such that $R_e = F \oplus V$ as vector spaces over $F$.
Put $T' = ( \oplus_{g \in G(e,e)} V \delta_g ) \oplus 
( \oplus_{g \in \mor(G) \setminus G_e } R_{c(g)} \delta_g)$.
It is clear that $S = T \oplus T'$ as additive groups.
What remains to show now is that $T'$ is a right (left) $T$-module.
We will only show the ``right'' part since the ``left'' part 
can be shown in an analogous manner.
To this end, take $f \in F$ and $g \in G_e$.
If $h \in G_e$ and $v \in V$, then
$(v \delta_h)(f \delta_g) = v \alpha_h(f) \delta_{hg} \in 
\oplus_{g \in G_e} V \delta_g$, since $F$ is $G_e$-invariant
and $V$ is a right $F$-vector space.
Suppose now that $h \in \mor(G) \setminus G_e$ and $r \in R_{c(h)}$.
If $d(h) \neq e$, then $(r \delta_h)(f \delta_g)=0 \in T'$.
If $d(h) = e$, then $c(h) \neq e$ and hence $hg \in \mor(G) \setminus G_e$.
Thus $(r \delta_h)(f \delta_g) \in \oplus_{g \in \mor(G) 
\setminus G_e } R_{c(g)} \delta_g \subseteq T'$.
\end{proof}

\begin{prop}\label{ONLYIF2} 
If $\alpha$ is a unital partial action, $R$ is semisimple and $R \star_\af G$ 
is left (right) artinian, 
then $R_g = \{ 0 \}$ for all but finitely many $g \in \mor(G)$.
\end{prop}

\proof 
Write $R = R_1 \oplus \cdots \oplus R_n$ as a direct sum of simple rings.
Put $\overline{n} = \{ 1,\ldots,n \}$.
Then, for every $g \in G$,
$R_g$ is the direct sum of a non-empty subset of $\{ R_i \}_{i \in \overline{n}}$. 
Define a new category $\overline{G}$ in the following way.
As objects we take all $(i,e) \in \overline{n} \times \ob(G)$
such that $R_i$ is a direct summand in $R_e$.
As morphisms we take all formal expressions
of the form ${}_j g_i$, for $i,j \in \overline{n}$ and $g \in \mor(G)$,
such that $R_i$ is a direct summand in $R_{g^{-1}}$,
$R_j$ is a direct summand in $R_g$ and $\alpha_g(R_i)=R_j$.
Define the domain $d$  and codomain $c$ by the relations
$d( {}_j g_i ) = (i,d(g))$ and $c( {}_j g_i ) = (j,c(g))$,
for ${}_j g_i \in \mor(\overline{G})$.
The composition of ${}_k g_j$ and ${}_j h_i$ in $\mor(\overline{G})$
is defined by the relation 
$({}_k g_j) ({}_j h_i) = {}_k (gh)_i$. 
Note that since $\alpha_e = {\rm id}_{R_e}$, 
for each $(i,e) \in \ob(\overline{G})$, 
${}_i e_i$ is the identity morphism at $(i,e)$.
It is clear that $\overline{G}$ is a groupoid, 
where for ${}_j g_i \in \mor(\overline{G})$,
one has  $({}_j g_i)^{-1} =  {}_i (g^{-1})_j$.
For each ${}_j g_i \in \mor(\overline{G})$,
put $R_{ {}_j g_i } = R_j=R_{ {}_j c(g)_j }$ and 
$\overline{\alpha}_{ {}_j g_i } = p_j \circ \alpha_g \circ q_i$,
where $p_j : R \rightarrow R_j$ is
the projection onto the $j$th coordinate 
and $q_i : R_i \rightarrow R$ is the injection defined by inclusion. 
Then $\overline{\alpha} = 
\{ \overline{\alpha}_{ {}_j g_i } : R_{ {}_i g^{-1}_j }  \rightarrow R_{ {}_j g_i } \}$
is a global action of $G$ on $R$ such that
$R \star_{\overline{\alpha}} \overline{G} = R \star_{\alpha} G$.
Note that $\overline{G}^{\#} = \overline{G}$.
Thus, by Proposition \ref{lemmaskewgroupoidring}, we get that 
$\mor(\overline{G})$ is a finite set.
Put ${\mathcal G} = \{ g \in \mor(G) \mid R_g \ \mbox{is non-zero} \}$.
For each $g \in {\mathcal G}$, fix $i(g),j(g) \in \overline{n}$
such that ${}_{j(g)} g_{i(g)} \in \overline{G}$.
It is clear that the map 
${\mathcal G} \ni g \mapsto {}_{j(g)} g_{i(g)} \in \overline{G}$
is injective. Thus, ${\mathcal G}$ is also finite.
\endproof

\subsection*{Proof of Theorem \ref{ParkGeneral}.}
The ``if'' statement follows from Proposition \ref{IF}.
Now we show the only ``if'' statement.
Suppose that $\alpha$ is a global and unital action of a groupoid $G$ on $R$
such that $R \star_{\alpha} G$ is left (right) artinian.
We wish to show that $R$ is left (right) artinian and that 
$R_g = \{ 0 \}$ for all but finitely many $g \in \mor(G)$.
The first claim follows from Proposition \ref{ONLYIF1}.
Now we show the second claim.
For a ring $S$, put $\overline{S} = S/A(S)$.
For a ring isomorphism $f : S \rightarrow T$,
let $\overline{f} : \overline{S} \rightarrow \overline{T}$
be the natural map.
Since $\alpha$ is a global  action of $G$
on $R$, the collection of maps 
$\overline{\alpha} := \{ \overline{\alpha}_g : \overline{R}_{g^{-1}} 
\rightarrow \overline{R}_g \}_{g \in \mor(G)}$ 
is a global  action of $G$ on 
$\overline{R}=\oplus_{e \in \ob(G)}\overline{R}_e$.
Since, for every $g \in \mor(G)$ with $R_g \neq \{ 0 \}$,
$R_g$ is unital,
the ring $\overline{R}_g$ is also non-zero and,
by Proposition \ref{semisimplicitySR}, semisimple.
Since $\overline{R} \star_{\overline{\alpha}} G$ is an 
epimorphic image of $R \star_{\alpha} G$, we get that
$\overline{R} \star_{\overline{\alpha}} G$ is artinian.
Now, by Proposition \ref{ONLYIF2}  we get that ${R}_g = {\mathcal S}(R_g)$ 
for all but finitely many $g \in \mor(G)$, and the claim follows from the fact that the ideals  ${R}_g$ are  unital.
\qed

\subsection*{Proof of Theorem \ref{ParkGeneralAlternative}.}
The ``if'' statement follows from Proposition \ref{IF}.
Now we show the ``only if'' statement.
Suppose that $\alpha$ is a partial action of a groupoid $G$ 
on an alternative ring $R$ such that 
$R \star_{\alpha} G$ is left (right) artinian.
We shall prove that $R$ is left (right) artinian
and that $R_g = \{ 0 \},$ for all but finitely many $g \in \mor(G)$.
The first claim follows from Proposition \ref{ONLYIF1}.
Now we show the second claim.
For a ring $S$, put $\overline{S} = S/{\mathcal J}(S)$.
For a ring isomorphism $f : S \rightarrow T$,
let $\overline{f} : \overline{S} \rightarrow \overline{T}$
be the natural map.
Take $g \in \mor(G)$. Then $\overline{R}_g$ 
can be viewed as an ideal of $\overline{R}_{c(g)}$.
Indeed, the inclusion $i : R_g \rightarrow R_{c(g)}$
induces a well defined ring homomorphism 
$i' : R_g \rightarrow \overline{R}_{c(g)}$.
By Proposition \ref{JIhereditary}, we get that $i'$ induces
a well defined injective ring homomorphism
$i'' : \overline{R}_g \rightarrow \overline{R}_{c(g)}$.
If we use $i''$ to identify each element of $\overline{R}_g$ with its
image in $\overline{R}_{c(g)}$, then $\overline{R}_g$ 
can be viewed as an ideal of $\overline{R}_{c(g)}$.
It is clear that axioms (i)-(v) hold for 
the collection of maps 
$\overline{\alpha} := \{ \overline{\alpha}_g : \overline{R}_{g^{-1}} 
\rightarrow \overline{R}_g \}_{g \in \mor(G)}$.
Since, for every $g \in \mor(G)$ with $R_g \neq \{ 0 \}$,
$R_g$ is unital,
the ring $\overline{R}_g$ is also non-zero and,
by Proposition \ref{semisimplicityJR}, semisimple.
Since $\overline{R} \star_{\overline{\alpha}} G$ is an 
epimorphic image of $R \star_{\alpha} G$, we get that
$\overline{R} \star_{\overline{\alpha}} G$ is artinian.
The claim now follows from Proposition \ref{ONLYIF2} 
and the fact that $\af$ is unital.
\qed

\section{Applications to Groupoid rings, Partial group rings and Matrix rings}\label{sec:AppMix}

Let $G$ be a groupoid.
Define an equivalence relation $\sim$ on $\ob(G)$
by saying that if $e,f \in \ob(G)$, then 
$e \sim f$ if there is $g \in \mor(G)$ with
$d(g)=e$ and $c(g)=f$.
Let $E$ be a set of representatives 
for the different
equivalence classes of $\sim$.
For each $e \in E$, let $R_e$ be a unital non-associative ring
and put $R = \oplus_{e \in \ob(G)} R_e$.
For each $g \in \mor(G)$, let $e(g)$ denote the unique element in $E$
such that $c(g) \sim e(g)$.
For each $g \in \mor(G)$, put $R_g = R_{e(g)}$
and let $\alpha_g : R_{g^{-1}} \rightarrow R_g$ 
be the identity map.
Then $\{ \alpha_g : R_{g^{-1}} \rightarrow R_g \}_{g \in \mor(G)}$
is a partial action of the groupoid $G$ on $R$.
The corresponding partial skew groupoid ring $R \star_{\alpha} G$
coincides with the groupoid ring $R[G]$ of $G$ over $R$.
Using the above notation, we get the following generalization of Connell's classical result for group rings
\cite[Theorem 1]{C}.

\begin{teo}\label{groupoidring}
The groupoid ring $R[G]$ is left (right) artinian
if and only if $\mor(G)$ is finite and for each $e \in E$,
the ring $R_e$ is left (right) artinian. 
\end{teo}

\begin{proof}
This follows immediately from Theorem \ref{ParkGeneral}.
\end{proof}

\begin{rem}
E. I. Zelmanov \cite{zelmanov1977} has shown that if the semigroup ring $R[S]$ is left (or right) artinian for an associative ring $R$,
then the semigroup $S$ must be finite.
In \cite{kozhukhov1998} I. B. Kozhukhov showed that Zelmanov's result can not be generalized 
to (non-associative) magma rings.
Theorem \ref{groupoidring} appears to be the first
generalization of Connell's result which applies to non-associative rings.
\end{rem}

Let $I$ be a non-empty set and suppose that $T$ 
is a unital non-associative ring.
Define a groupoid ${\mathcal I}$ in the following way.
As objects of ${\mathcal I}$ we take the elements of $I$.
As morphisms of ${\mathcal I}$ we take the elements of $I \times I$.
Take $i,j,k \in I$.
We put $d(i,j) = j$ and $c(i,j)=i$.
The composition of $(i,j)$ with $(j,k)$, 
denoted $(i,j)(j,k)$, is defined to be $(i,k)$.
The equivalence relation $\sim$, from above, 
only has one equivalence class for the groupoid ${\mathcal I}$.
In other words, ${\mathcal I}$ is connected.
So if we let $R_e = T$, for $e \in \ob(G)$, 
then the skew groupoid ring $R \star_{\alpha} {\mathcal I}$
coincides with the generalized matrix ring $M_I(T)$.
With the above notation, we get the following result.

\begin{cor}\label{cor:genmatrixring}
The generalized matrix ring $M_I(T)$ is left (right) artinian
if and only if $I$ is finite and $T$ is left (right) artinian.
\end{cor}

\begin{proof}
This follows immediately from Theorem \ref{groupoidring}.
\end{proof}

Given a field $K$ and a group $G$
one may define the \emph{partial group algebra} $K_{\parpar}[G]$ (see e.g. \cite{DokuExelPicc2000}).

\begin{cor}\label{cor:PGA}
Let $K$ be a field.
The partial group algebra $K_{\parpar}[G]$ is artinian
if and only if $G$ is a finite group.
\end{cor}

\begin{proof}
We first show the ``if'' statement.
Using that $G$ is a finite group we get that the Brandt groupoid of $G$, denoted by $\Gamma$,
is finite.
By \cite[Corollary 2.7]{DokuExelPicc2000}, the partial group algebra $K_{\parpar}[G]$
is isomorphic (as a $K$-algebra) to the groupoid algebra $K[\Gamma]$.
The desired conclusion now follows immediately from Theorem \ref{groupoidring}.

Now we show the ``only if'' statement.
By \cite[Theorem 6.9]{DokuchaevExel05}, $K_{\parpar}[G]$ is isomorphic (as a $K$-algebra)
to a certain partial skew group ring $\mathcal{A} \rtimes_{\alpha^{\tilde{\pi}}} G$
associated to a partial action $\{\{\alpha^{\tilde{\pi}}_g\}_{g\in G}, \{ D_g\}_{g\in G} \}$.
For each $g\in G$, $D_g$ contains the (non-zero) identity element $[g][g^{-1}]$ (where $[g],[g^{-1}]$ are elements of Exel's semigroup).
Hence, if $K_{\parpar}[G]$ is artinian, then, by Theorem \ref{ParkGeneral}, $D_g =\{0\}$ for all but finitely many $g\in G$.
This shows that $G$ is finite.
\end{proof}

\begin{remark}
The ``only if'' statement of Corollary \ref{cor:PGA} can also be obtained through Zelmanov's theorem \cite{zelmanov1977}.
Indeed, $K_{\parpar}[G]$ can be viewed as a semigroup algebra by Exel's semigroup. 
\end{remark}

\section{Applications to Leavitt path algebras}\label{sec:LPA}

A directed graph $E=(E^0,E^1,r,s)$ consists of two countable sets $E^0$, $E^1$ and maps $r,s : E^1 \to E^0$. The elements of $E^0$ are called \emph{vertices} and the elements of $E^1$ are called \emph{edges}. If both $E^0$ and $E^1$ are finite sets, then we say that $E$ is \emph{finite}.
A vertex which emits no edge is called a \emph{sink}. A vertex $v\in E^0$ such that $|s^{-1}(v)|=\infty$ is called an \emph{infinite emitter}.
A \emph{path} $\mu$ in 
$E$ is a sequence of edges $\mu = \mu_1 \ldots \mu_n$ such that $r(\mu_i)=s(\mu_{i+1})$ for $i\in \{1,\ldots,n-1\}$. In such a case, $s(\mu):=s(\mu_1)$ is the \emph{source} of $\mu$, $r(\mu):=r(\mu_n)$ is the \emph{range} of $\mu$ and $n$ is the \emph{length} of $\mu$.
Recall that a path $\mu$ is called a \emph{cycle} if $s(\mu)=r(\mu)$ and $s(\mu_i)\neq s(\mu_j)$ for every $i\neq j$.
A graph $E$ without cycles is said to be \emph{acyclic}.

\begin{defi}[Leavitt path algebra \cite{AAP}]
Let $E$ be any directed graph and let $K$ be any field.
The \emph{Leavitt path $K$-algebra $L_K(E)$ of $E$ with coefficients in $K$} is the
$K$-algebra generated by a set $\{v \mid v\in E^0\}$ of pairwise orthogonal idempotents, together with a set of variables $\{f \mid f\in E^1\} \cup
\{f^* \mid f\in E^1\}$, which satisfy the following  relations:
\begin{enumerate}
\item $s(f)f=fr(f)=f$, for all $f\in E^1$;
\item $r(f)f^*=f^*s(f)=f^*$, for all $f\in E^1$;
\item $f^*f'=\delta_{f,f'}r(f)$, for all $f,f'\in E^1$;
\item $v=\sum_{ \{ f\in E^1 \mid s(f)=v \} } ff^*$, for every $v\in E^0$ for which $s^{-1}(v)$ is non-empty and finite.
\end{enumerate}
\end{defi}

In \cite{GR} D. Gon\c{c}alves and D. Royer have shown that the Leavitt path algebra $L_K(E)$ is isomorphic (as a $K$-algebra) to a partial skew group ring $D(X) \star_\alpha \F$, where $D(X)$ is a certain commutative $K$-algebra and $\F$ is the free group generated by $E^1$.
For the benefit of the reader we shall briefly recall the construction of $D(X) \star_\alpha \F$.

Let $W$ denote the set of all finite paths in $E$, and let $W^{\infty}$ denote the set of all infinite paths in $E$.
Let $\F$ denote the free group genereated by $E^1$.
We are going to define a partial action of $\F$ on the set
\begin{displaymath}
	X=\{\xi\in W \mid r(\xi) \text{ is a sink }\}\cup \{v\in E^0 \mid v\text{ is a sink }\}\cup W^{\infty}.
\end{displaymath}
For each $g\in \F$, let $X_g$ be defined as follows:
\begin{itemize}
\item $X_e:=X$, where $e$ is the identity element of $\F$. 

\item $X_{b^{-1}}:=\{\xi\in X \mid s(\xi)=r(b)\}$, for all $b\in W$. 

\item $X_a:=\{\xi\in X \mid \xi_1\xi_2...\xi_{|a|}=a\}$, for all $a\in W$.

\item $X_{ab^{-1}}:=\{\xi\in X \mid \xi_1\xi_2...\xi_{|a|}=a\}=X_a$, for $ab^{-1}\in \F$ with $a,b\in W$, $r(a)=r(b)$ and $ab^{-1}$ in its reduced form.

\item $X_g:=\emptyset$, for all other $g \in \F$.

\end{itemize}

Let  $\theta_e:X_e\rightarrow X_e$ be the identity map. For $b\in W$,  $\theta_b:X_{b^{-1}}\rightarrow X_b$ is defined by $\theta_b(\xi)=b\xi$ and $\theta_{b^{-1}}:X_b\rightarrow X_{b^{-1}}$ by
$\theta_{b^{-1}}(\eta)= \eta_{|b|+1}\eta_{|b|+2}\ldots$ if $r(b)$ is not a sink, and by $\theta_{b^{-1}}(b)=r(b)$ if $r(b)$ is a sink. Finally, for $a,b\in W$ with $r(a)=r(b)$ and $ab^{-1}$ in reduced form, $\theta_{ab^{-1}}:X_{ba^{-1}}\rightarrow X_{ab^{-1}}$ is defined by $\theta_{ab^{-1}}(\xi)=a\xi_{(|b|+1)}\xi_{(|b|+2)}\ldots$, with inverse 
$\theta_{ba^{-1}}:X_{ab^{-1}}\rightarrow X_{ba^{-1}}$ defined by $\theta_{ba^{-1}} (\eta)=b\eta_{(|a|+1)}\eta_{(|a|+2)}\ldots$ .

One may easily verify that $\{\{X_g\}_{g\in \F}, \{\theta_g\}_{g\in\F}\}$ is a partial action on the set level which induces a partial action on the algebra level.
Indeed, we may define a partial action $\{\{F(X_g)\}_{g\in \F}, \{\alpha_g\}_{g\in\F}\}$, where, for each $g\in \F$, $F(X_g)$ denotes the algebra of all functions from $X_g$ to $K$, and $\alpha_g:F(X_{g^{-1}})\rightarrow F(X_g)$ by putting $\alpha_g(f)=f\circ \theta_{g^{-1}}$.
Based on this partial action, we define another partial action in the following way:

For each $g\in \F$, and for each $v\in E^0$, define the characteristic maps $1_g:=\chi_{X_g}$ and $1_v:=\chi_{X_v}$, where $X_v=\{\xi\in X \mid s(\xi)=v\}$. Notice that $1_g$ is the identity element of $F(X_g)$. Finally, let
\begin{displaymath}
	D(X)=D_e=\Kspan\{\{1_g \mid g\in \F\setminus\{0\}\}\cup\{1_v \mid v\in E^0\}\},
\end{displaymath}
(where $\Kspan$ means the $K$-linear span) and, for each $g\in \F\setminus\{0\}$, let $D_g\subseteq F(X_g)$ be defined as $1_g D_e$, that is,
\begin{displaymath}
	D_g=\Kspan\{1_g1_h \mid h\in \F\}.
\end{displaymath}
By \cite[Lemma 2.4]{GR}, $D(X)$ is a $K$-algebra and $D_g$, for $g\in \F$, is an ideal of $D(X)$.
Using that $\alpha_g(1_{g^{-1}}1_h)=1_g1_{gh}$ (see \cite[Lemma 2.6]{GR}), for each $g\in \F$, 
it is clear that the restriction of $\alpha_g$ to $D_{g^{-1}}$ is a bijection onto $D_g$.
By abuse of notation this restriction map will also be denoted by $\alpha_g$.
Clearly, $\alpha_{g}:D_{g^{-1}}\rightarrow D_g$ is an isomorphism of $K$-algebras and, furthermore, $\{\{\alpha_g\}_{g\in \F}, \{D_g\}_{g\in \F}\}$ is a partial action. By \cite[Proposition 3.2]{GR} the map $\varphi : L_K(E) \to D(X) \star_\alpha \F$
defined by $\varphi(f)=1_f \delta_f$, $\varphi(f^*)=1_{f^{-1}}\delta_{f^{-1}}$, for all $f\in E^1$, and $\varphi(v)=1_v \delta_e$, for all $v\in E^0$, is an isomorphism of $K$-algebras.

\begin{teo}\label{LPAthm}
Let $K$ be a field and let $E$ be a directed graph.
Consider $L_K(E)$, the Leavitt path $K$-algebra of $E$ with coefficients in $K$.
The following five assertions are equivalent:
\begin{enumerate}[{\rm (i)}]
	\item\label{LPAthm:Econd} $E$ is finite and acyclic;
	\item\label{LPAthm:leftLPAartin} $L_K(E)$ is left artinian;
	\item\label{LPAthm:rightLPAartin} $L_K(E)$ is right artinian;
	\item\label{LPAthm:LPAartin} $L_K(E)$ is artinian;
	\item\label{LPAthm:LPAuss} $L_K(E)$ is unital and semisimple.
\end{enumerate}
\end{teo}

\begin{proof}
Throughout this proof we are going to make use of the fact that
the Leavitt path algebra $L_K(E)$ is isomorphic to the partial skew group ring $D(X) \star_\alpha \F$, as described above.

\eqref{LPAthm:Econd}$\Rightarrow$\eqref{LPAthm:leftLPAartin}:
If we can show that $D(X)$ is left artinian and that $D_g = \{0\}$ for all but finitely many $g\in \mathbb{F}$, then it follows by Theorem \ref{ParkGeneral} that $D(X) \star_\alpha \mathbb{F}$ is an artinian ring.

By the finiteness and acyclicity of $E$
we immediately conclude that there are no infinite paths in $E$
and that $W$, the set of all finite paths in $E$, is a finite set.
Hence, using the notation of \cite{GR}
we see that $X=\{\xi \in W \mid r(\xi) \text{ is a sink}\} \cup \{v\in E^0 \mid v \text{ is a sink}\}$
is a finite set.
Using that $W$ is finite, we conclude that $X_g = \emptyset$ for all but finitely many $g\in \mathbb{F}$.
Hence, $1_g =0$ for all but finitely many $g\in \mathbb{F}$. Using this, we immediately conclude that $D_g=\{0\}$ for all but finitely many $g\in \mathbb{F}$, and moreover, we notice that it turns $D(X)$ into a finite-dimensional $K$-vector space.
Hence, $D(X)$ is artinian.

\eqref{LPAthm:leftLPAartin}$\Rightarrow$\eqref{LPAthm:Econd}:
Suppose that $L_K(E) \cong D(X) \star_\alpha \F$ is left artinian.
By Theorem \ref{ParkGeneralAlternative}, we get that $D_g=\{0\}$ for all but finitely many $g\in \F$.

We claim that there can be no infinite path in $E$, i.e. $W^\infty = \emptyset$.
Clearly, every cycle gives rise to an infinite path. Thus, if we assume that the claim holds, then $E$ must be acyclic.
Now we prove the claim.
Seeking a contradiction, suppose that $E$ contains an infinite path $\xi = \xi_1 \xi_2 \xi_3 \ldots$.
By taking (finite) initial subpaths of $\xi$, which are elements of $W$ and hence also of $\F$,
we can form an infinite chain of nested subsets
\begin{displaymath}
	X_{\xi_1} \supseteq X_{\xi_1 \xi_2} \supseteq X_{\xi_1 \xi_2 \xi_3} \supseteq \ldots
\end{displaymath}
which are all non-empty since they contain $\xi$.
Thus, the ideals
\begin{displaymath}
	D_{\xi_1} , \quad D_{\xi_1 \xi_2} , \quad D_{\xi_1 \xi_2 \xi_3}, \quad \ldots
\end{displaymath}
are all non-zero. Hence, $D_g \neq \{0\}$ for infinitely many $g\in \F$.
This is a contradiction.

Now we show that $E$ must be finite. Seeking a contradiction, suppose that $E^0=\{v_1,v_2,v_3,\ldots\}$ is infinite.
Notice that $L_K(E)=\oplus_{v\in E^0} L_K(E)v$
and consider the following
descending chain of left ideals of $L_K(E)$.
\begin{displaymath}
	L_K(E) \supseteq \oplus_{v\in E^0 \setminus \{v_1\}} L_K(E)v \supseteq \oplus_{v\in E^0 \setminus \{v_1,v_2\}} L_K(E)v \supseteq \oplus_{v\in E^0 \setminus \{v_1,v_2,v_3\}} L_K(E)v \supseteq \ldots
\end{displaymath}
It is easy to see that the above chain never stabilizes,
using the fact that every pair of vertices in $E^0$ are orthogonal idempotents.
Hence, $L_K(E)$ is not left artinian. This is a contradiction.
We conclude that $E^0$ is finite.

We now proceed to show that $E^1$ must be finite.
Using that $E^0$ is finite, it is sufficient to show that $E^0$ contains no infinite emitter.
Seeking a contradiction, suppose that there is a vertex $v\in E^0$
which is an infinite emitter.
Since $E^0$ is finite, there must exist some $u\in E^0$ such that
the set $I=\{e \in E^1 \mid s(e)=v \text{ and } r(e)=u\}$ is infinite.

{\bf Case 1:} ($u$ is a sink)\\
Take $e\in I \subseteq W$ and consider the set
$X_{e^{-1}}=\{\xi \in X \mid s(\xi)=u\}$ which is non-empty since it contains
$u$.
Hence, $D_{e^{-1}}$ is non-zero for (infinitely many) $e\in I$.
This is a contradiction.

{\bf Case 2:} ($u$ is not a sink)\\
We have already shown that $E$ contains no infinite path, i.e. $W^\infty=\emptyset$.
Hence, there must exist at least one path from $u$ to a sink $w$. Let us call it $\eta$.
Take $e\in I$ and consider the set
$X_{e^{-1}}=\{\xi \in X \mid s(\xi)=u\}$ which is non-empty since it contains
$\eta$.
Hence, $D_{e^{-1}}$ is non-zero for (infinitely many) $e\in I$.
This is a contradiction.

This shows that $E^1$ (and hence also $E$) is finite.

\eqref{LPAthm:leftLPAartin}$\Leftrightarrow$\eqref{LPAthm:rightLPAartin}$\Leftrightarrow$\eqref{LPAthm:LPAartin}:
This follows immediately from Theorem \ref{ParkGeneralAlternative}, using the fact that $D(X)$ is commutative.

\eqref{LPAthm:LPAartin}$\Rightarrow$\eqref{LPAthm:LPAuss}: By \cite[Proposition 6.3]{AAP} we know that $L_K(E)$ is semiprimitive, i.e. $J(L_K(E))=\{0\}$. Hence, if $L_K(E)$ is left artinian it must also be semisimple (see e.g. \cite[Theorem 4.14]{LAM}).
Using that \eqref{LPAthm:LPAartin}$\Leftrightarrow$\eqref{LPAthm:Econd} we get that $E$ is finite and hence, in particular, $L_K(E)$ is unital.

\eqref{LPAthm:LPAuss}$\Rightarrow$\eqref{LPAthm:LPAartin}: This is clear.
\end{proof}

\begin{remark}
The essence of the above result has previously been shown in \cite{AAPSM}
using a different technique.
\end{remark}

Recall that a subset $H\subseteq E^0$ is said to be \emph{hereditary} if for any $e\in E^1$ we have that $s(e)\in H$ implies $r(e)\in H$. A hereditary subset $H\subseteq E^0$ is called \emph{saturated} if whenever $0<\# s^{-1}(v) < \infty$, then $\{r(e)\in H \mid e\in E^1 \text{ and } s(e)=v\} \subseteq H$ implies $v\in H$.

\begin{cor}\label{cor:finacyclic}
If $E$ is a finite and acyclic graph such that
the only hereditary and saturated subsets of $E^0$ are $\emptyset$ and $E^0$,
then the Leavitt path algebra $L_K(E)$ is a simple unital artinian ring.
\end{cor}

\begin{proof}
Since $E$ is acyclic it trivially satisfies condition (L), see \cite{GOR2014}.
By \cite[Theorem 3.5]{GOR2014} we conclude that $L_K(E)$ is simple.
We get from Theorem \ref{LPAthm} that $L_K(E)$ is a simple unital artinian ring.
\end{proof}

\section{Morita equivalence, globalization, Noetherianity and Maschke-type results}\label{sec:Morita}

In this section, we use globalization and Morita 
equivalence to deduce necessary and sufficient conditions
for the artinianity, noetherianity and two
Maschke-type reults.
We adapt the approach taken in \cite{FL}
to the groupoid situation.
To this end, we recall the following.

\begin{defi}
Following \cite{BP}, we say that a global action 
$\bt = \{ \bt_g : T_{g^{-1}} \rightarrow T_g \}_{g\in \mor(G)}$ 
of a groupoid $G$ on an associative ring $T$ is
a {\it globalization} of a partial action 
$\af = \{ \af_g : R_{g^{-1}} \rightarrow R_g \}_{g\in \mor(G)}$ 
of $G$ on an associative ring $R$ if, for each $e\in \ob(G)$, 
there exists a ring monomorphism $\psi_e : R_e \to T_e$ such that,
for every $g \in {\rm mor}(G)$ and every $a \in R_{g^{-1}}$,
the following hold:
\begin{enumerate}\renewcommand{\theenumi}{\roman{enumi}}   \renewcommand{\labelenumi}{(\theenumi)}
	\item $\psi_e(R_e)\unlhd T_e$, i.e. $\psi_e(R_e)$ is a two-sided ideal of $T_e$;
	\item $\psi_{c(g)}(R_{g})=\psi_{c(g)}(R_{c(g)})\cap\bt_g(\psi_{d(g)}(R_{d(g)}))$;
	\item $\bt_g(\psi_{d(g)}(a))=\psi_{c(g)}(\af_g(a))$;
	\item and $T_g=\sum\limits_{c(h)=c(g)}\bt_h(\psi_{d(h)}(R_{d(h)}))$.
\end{enumerate}
\end{defi}

\begin{remark}
Throughout this section, let 
$\af = \{ \af_g : R_{g^{-1}} \rightarrow R_g \}_{g\in \mor(G)}$ 
be a unital partial action of groupoid $G$ on a ring $R$. 
It follows from \cite[Theorem 2.1]{BP} that $\af$ admits a globalization 
$\bt = \{ \bt_g : T_{g^{-1}} \rightarrow T_g \}_{g\in \mor(G)}$.
From \cite[Remark 2.3]{BP}, we may, for the rest of the section, 
fix $\alpha$ and $\beta$ so that 
$\psi_e=\id_{R_e}$. Then, $R_e\unlhd T_e$ and $T_e1_{R_e}=R_e,$ for any $e\in \ob(G).$ 
We shall also assume that $T = \bigoplus_{e\in \ob(G)}T_e.$
\end{remark}

\begin{defi}
We say that $\alpha$ is \emph{of finite type}, if for any  
$e\in \ob(G)$ there are $g_1,\ldots, g_n\in G(-,e)$ such that 
$R_{c(g)}=\sum_{i=1}^n R_{gg_i},$ for any $g\in G(e,-),$ where
$G(-,e)=\{g\in G\mid c(g)=e\}$ and $G(e,-)=\{g\in G\mid d(g)=e\}.$
\end{defi}

\begin{prop}\label{fintype}
The following three assertions are equivalent:
\begin{enumerate}[{\rm (i)}]
\item \label{ft}$\alpha$ is of finite type;
\item \label{finsum} For any $e\in \ob(G),$ there exists $g_1,\ldots, g_n\in G(-,e)$ such that $T_e=\sum\limits_{i=1}^n\bt_{g_i}(R_{d(g_i)});$
\item \label{unit}For any $e\in \ob(G),$ the ring $T_e$ is unital.
\end{enumerate}
\end{prop}

\begin{proof}
\eqref{ft}$\Rightarrow$\eqref{finsum}:  Let  $e\in \ob(G).$ Then by (iv) in the definition of a globalization one has that  $T_e=\sum\limits_{c(h)=e}\bt_h(R_{d(h)}).$ Since  $\alpha$ is of finite type there are $g_1,\ldots, g_n\in G(-,e)$ such that $R_{d(h)}=\sum_{i=1}^{n} R_{h\m g_i},$ for each $h\in G$ with $c(h)=e.$ Then  $\bt_h(R_{d(h)})=\sum_{i=1}^{n}\bt_h( R_{h\m g_{i}}).$ From this we get that
\begin{align*}
\sum\limits_{i=1}^{n}\bt_{g_i}(R_{d(g_i)})\subseteq T_e=\sum\limits_{c(h)=e}\bt_h(R_{d(h)})=
\sum\limits_{c(h)=e}\sum\limits_{i=1}^{n}\bt_h( R_{h\m g_{i}})
&=\sum\limits_{i=1}^{n} \sum\limits_{c(h)=e}\bt_h( R_{h\m g_{i}}) \\
&\subseteq \sum\limits_{i=1}^{n}\bt_{g_i}(R_{d(g_i)}).
\end{align*}
using that
	$\bt_h( R_{h\m g_{i}})=\bt_h(R_{c(h\m g_i)})\cap \bt_h\bt_{h\m g_i}(R_{d(h\m g_i)}) \subseteq \bt_{g_i}(R_{d(g_i)}),$
for each $h\in G(-,e)$.
Since $ \sum\limits_{i=1}^{n}\bt_{g_i}(R_{d(g_i)})\subseteq T_e$  we conclude that $T_e=\sum\limits_{i=1}^{n}\bt_{g_i}(R_{d(g_i)}) ,$ as desired.

\eqref{finsum}$\Rightarrow$\eqref{unit} This follows from the fact that $T_e$ is a finite sum of unital rings (see \cite[Lemma 4.4]{DokuchaevExel05}).

\eqref{unit}$\Rightarrow$\eqref{ft} There are $h_1,\ldots, h_n\in G(-,e)$ such that $1_{T_e}=\sum\limits_{i=1}^{n}\bt_{h_i}(r_i),$ where $r_i\in R_{d(h_i)},$ for $i \in \{1,\ldots,n\}$. Thus, for any $g\in G(e,-)$ we get that $1_{T_{c(g)}}=\bt_g(1_{T_{e}})=\sum\limits_{i=1}^{n}\bt_{gh_i}(r_i)$.
Since $R_{c(g)}\unlhd T_{c(g)}$ one has that 
$1_{R_{c(g)}}=\sum\limits_{i=1}^{n}\bt_{gh_i}(r_i)1_{R_{c(g)}}\in \sum\limits_{i=1}^{n}[\bt_{gh_i}(R_{d(gh_i)})\cap R_{c(gh_{i})} ]= \sum\limits_{i=1}^{n}\bt_{gh_i}(R_{gh_i})$ which is an ideal of $R_{c(g)}.$ Hence  $R_{c(g)}=\sum\limits_{i=1}^{n}\bt_{gh_i}(R_{gh_i})$ and $\af$ is of finite type.
\end{proof}

\begin{remark}
The 
partial skew groupoid ring $R \star_\alpha G$ 
is, in the sense of \cite{lundstrom2004}, a groupoid graded ring.
Therefore, from \cite[Proposition 2.1.1]{lundstrom2004}
(or in a more general context \cite[Proposition 5]{lundstrom2006}),
it follows that $R \star_\alpha G$
is unital if ${\rm ob}(G)$ is finite.
In that case, if $\alpha$ is of finite type,
it follows from Proposition \ref{fintype} and \cite[Theorem 3.2]{BP}, 
that $R \star_\alpha G$ and $T \star_{\beta} G$ are Morita equivalent.
\end{remark}

\begin{cor}
If $\ob(G)$ is finite, then $T$ is right/left noetherian (artinian) 
if and only $R$ is right/left noetherian (artinian) and $\alpha$ is of finite type.
\end{cor}

\begin{proof}
Suppose that $\ob(G)$ is finite. 
First we show the ''if'' part.
Suppose that $R$ is right/left noetherian (artinian) and $\alpha$ is of finite type. 
Then by Proposition \ref{fintype}(ii), each $T_e$ is a finite sum of 
right/left noetherian (artinian)
$T_e$-modules, which implies that $T_e$ is right/left noetherian (artinian), for all $\ob(G).$ 
Since  $T = \bigoplus_{e\in \ob(G)} T_e$, it follows from Proposition \ref{directsumartinian},
that $T = \bigoplus_{e\in \ob(G)} T_e$ is right/left noetherian (artinian). 
Now we show the ''only if'' part.
If $T$ is right/left noetherian (artinian), so is $R,$ because $R$ is an ideal of $T.$
It remains to show that $\alpha$ is of finite type. 
Seeking a contradiction, suppose that there is $e\in \ob(G)$ 
such that $T_e$ is not a finite sum of ideals $\bt_{g}(R_{d(g)})$ with $g\in G(-,e)$. 
Then there is an infinite ascending sequence of these sums in which any term is 
generated by a central idempotent of $T_e$ . 
This contradicts the fact that $T_e$ is noetherian.
Also, the annihilators of these idempotents give an infinite descending
sequence of ideals of $T_e$, which contradicts the fact that $T_e$ is right artinian.
\end{proof}

\begin{cor}\label{art}
If $\ob(G)$ is finite, then $R \star_\alpha G$ is left/right artinian  
if and only if $T \star_\beta G$ is left/right artinian 
and $\alpha$ is of finite type.
In that case, both $R$ and $T$ are left/right artinian.
\end{cor}

\begin{proof}
Suppose that $R \star_\alpha G$ is left/right artinian. 
By Proposition \ref{ONLYIF1}, we get that
$R_g = \{ 0 \}$, for all but finitely many $g \in {\rm mor}(G)$.
Thus, from (iv) in the definition of a globalization, we get that
each $T_g$ is a finite sum of unital rings.
From Proposition \ref{fintype}(iii), we thus get that $\alpha$ is of finite type.
Thus from Morita equivalence it follows that $T \star_\beta G$ is left/right artinian.
Conversely, if $T \star_\beta G$ is left/right artinian and $\alpha$ is of finite type, 
then again by Morita equivalence, we get that 
$R \star_\alpha G$ is left/right artinian.
The last part follows from Proposition \ref{ONLYIF1}.
\end{proof}

\begin{cor}
If $\ob(G)$ is finite and $\alpha$ is of finite type, then 
$R \star_\alpha G$ is left/right noetherian if and only if 
$T \star_\beta G$ is left/right noetherian.
In that case, both $R$ and $T$ are left/right noetherian. 
\end{cor}

\begin{proof}
This follows from Morita equivalence and Proposition \ref{ONLYIF1}.
\end{proof}

We shall now prove some Maschke-type results
for partial skew groupoid rings associated with unital partial actions of groupoids on rings
(see Theorem~\ref{teo:Maschke} and Corollary~\ref{cor:Maschke}).

\begin{lema}\label{ss}
If $\ob(G)$ is finite, then $T$ is semisimple if and only 
if $R$ is semisimple and $\alpha$ is of finite type.
\end{lema}

\begin{proof}
Suppose that $T$ is semisimple. Then $T$ has an identity.
Thus, from Proposition~\ref{fintype}, it follows that 
$\alpha$ is of finite type.
Also, since $R$ is an ideal of $T$, it follows that $R$ is semisimple.
Conversely, suppose that $R$ is semisimple and
that $\alpha$ is of finite type. Then each $T_e$, for $e \in {\rm ob}(G)$, is a finite sum 
of semisimple rings. Thus, since ${\rm ob}(G)$ is finite, $T$ is semisimple.
\end{proof}

\begin{lema}\label{invertible}
For any positive integer $m$, the following hold:
\begin{itemize}

\item[(a)] $R$ has $m$-additive torsion
if and only if $T$ has $m$-additive torsion;

\item[(b)] if $T$ is unital, then $m$ is 
invertible in $R$ if and only if $m$ is invertible in $T$. 

\end{itemize}
\end{lema}

\begin{proof}
Analogous to the proof of \cite[Proposition 1.22]{FL}.
\end{proof}

\begin{teo}\label{teo:Maschke}
If ${\rm mor}(G)$ is finite,
$R$ is semisimple and for every $e\in \ob(G)$, 
$| G_e |$ is invertible in $R$, then $R \star_\alpha G$ is semisimple.
\end{teo}

\begin{proof}
From Lemma \ref{ss} it follows that $T$ is semisimple.
From Lemma \ref{invertible}(b), we get that 
for every $e\in \ob(G)$, $| G_e |$ is invertible in $T$.
Thus, since $T \star_\beta G$ is strongly graded by $G$,
we get, from \cite[Proposition 10(b)]{lundstrom2006},
that $T \star_\beta G$ is semisimple.
From Morita equivalence, we get that $R \star_\alpha G$ is semisimple.
\end{proof}

\begin{defi}
Suppose that ${\rm mor}(G)$ is finite.  
The \emph{trace map} was defined in \cite{BP} as
\begin{displaymath}
	\mathrm{tr}_\alpha\colon R \ni x\mapsto \sum_{g\in \mor(G)}\af_g(x1_{g\m})\in R^\alpha,
\end{displaymath}
where $R^\alpha=\{x\in R\mid \af_g(x1_{g\m})=x1_g,\, \text{for all}\, g\in G\}.$
\end{defi}

\begin{teo}\label{teo:modulessubmodules}
Suppose that ${\rm mor}(G)$ is finite.
Let $V$ be a left $R \star_\alpha G$-module and let $W$ be a submodule of $V$. 
If $\mathrm{tr}_\alpha(1_R)$ is invertible in $R$ and 
$W$ is a direct summand of $V$ as an $R$-module, 
then $W$ is a direct summand of $V$ as an $R \star_\alpha G$-module.
\end{teo}

\begin{proof}
We will only provide a sketch of the full proof
as it is analogous to the proof of \cite[Theorem 3.2]{FL}.
Suppose that $\pi : V \to W$ is an $R$-projection.
Put $l=(\mathrm{tr}_\alpha(1_R))^{-1}$.
For any $v\in V$, we define
$\psi(v)= l \sum_{g\in \mor(G)} 1_{g^{-1}} \delta_{g^{-1}} \pi(1_g \delta_g v)$.
Clearly, this yields a well-defined map $\psi : V \to W$
and one can check that $\psi$ is in fact a left $R \star_\alpha G$-module homomorphism.
We notice that $\psi(w)=w$ for any $w\in W$. Thus, $\psi$ is a projection onto $W$.
\end{proof}
The following is an immediate consequence of Theorem~\ref{teo:modulessubmodules}.

\begin{cor}\label{cor:Maschke}
If ${\rm mor}(G)$ is finite, $R$ is semisimple and 
$tr_\alpha(1_R)$ is invertible in $R$, then $R \star_\alpha G$ is semisimple.
\end{cor}

\end{document}